\documentclass[a4paper,12pt,english]{smfart}

\usepackage{amssymb}
\usepackage{url}
\RequirePackage{mathrsfs}
\DeclareSymbolFont{rsfscript}{OMS}{rsfs}{m}{b}
\DeclareSymbolFontAlphabet{\mathrsfs}{rsfscript}
\usepackage[utopia,expert]{mathdesign}
\usepackage{lscape}
\usepackage{calligra}
\usepackage[T1]{fontenc}

\usepackage{frcursive}

\usepackage{palatino}
\usepackage{rotating}
\usepackage{graphicx}

\usepackage{enumerate}

\usepackage{color}
\definecolor{shadecolor}{gray}{0.95}

\def\longmapright#1{\hspace{0.3em}\smash{
     \mathop{\longrightarrow}\limits^{#1}}\hspace{0.3em}}
\def\longisom{{\longmapright{\sim}}}

\def\proofname{D\'emonstration}

\input xypic
\xyoption{all}
\xyoption{arc}

\newtheorem{theo}{Theorem}[section]
\def\thetheo{\thesection.\arabic{theo}}
\newtheorem{prop}[theo]{Proposition}

\newtheorem{lem}[theo]{Lemma}

\newtheorem{coro}[theo]{Corollary}

\renewcommand\theequation{\thetheo}
\def\equat{\refstepcounter{theo}\begin{equation}}
\def\endequat{\end{equation}}

\renewcommand\thetable{\thetheo}

\renewcommand\thesection{\arabic{section}}
\renewcommand\thesubsection{\thesection.\Alph{subsection}}

\def\contentsname{Table des mati\`eres}
\def\listfigurename{Liste des figures}
\def\listtablename{Liste des tables}
\def\appendixname{Appendice}

\def\refname{R\'ef\'erences}

\def\IG{{\mathfrak I}}

    \def\ZM{{\mathbb{Z}}}

  \def\cb{{\mathbf c}}  \def\CC{{\mathcal{C}}}
    
    \def\EC{{\mathcal{E}}}

  \def\hb{{\mathbf h}}  \def\HC{{\mathcal{H}}}

    \def\KC{{\mathcal{K}}}
    \def\LC{{\mathcal{L}}}
    \def\MC{{\mathcal{M}}}

    \def\PC{{\mathcal{P}}}
    
    \def\RC{{\mathcal{R}}}

    \def\VC{{\mathcal{V}}}
    
    \def\XC{{\mathcal{X}}}

          \def\aov{{\overline{a}}}

\def\Mov{{\overline{M}}}          \def\mov{{\overline{m}}}

\def\g{\gamma}
\def\G{\Gamma}

\def\ph{\varphi}

\def\L{\Lambda}

\def\O{\Omega}

\def\s{\sigma}

\DeclareMathOperator{\maps}{{\mathrm{Maps}}}

\def\to{\rightarrow}
\def\longto{\longrightarrow}

\def\fonctio#1#2#3#4{\begin{array}{ccc}
{#1} & \longto & {#2} \\
{#3} & \longmapsto & {#4} 
\end{array}}

\def\vide{\varnothing}
\def\DS{\displaystyle}

\def\finl{~$\blacksquare$}

\def\lexp#1#2{\kern\scriptspace\vphantom{#2}^{#1}\kern-\scriptspace#2}
\def\le{\hspace{0.1em}\mathop{\leqslant}\nolimits\hspace{0.1em}}
\def\ge{\hspace{0.1em}\mathop{\geqslant}\nolimits\hspace{0.1em}}

\mathchardef\inferieur="321E
\mathchardef\superieur="321F

\def\eqna{\begin{eqnarray*}}
\def\endeqna{\end{eqnarray*}}

\def\itemth#1{\item[${\mathrm{(#1)}}$]}

\catcode`\@=11
\long\def\@car#1#2\@nil{#1}
\long\def\@first#1#2{#1}
\long\def\@second#1#2{#2}
\long\def\ifempty#1{\expandafter\ifx\@car#1@\@nil @\@empty
  \expandafter\@first\else\expandafter\@second\fi}
\catcode`\@=12

\renewcommand{\tocsection}[3]{%
  {\bf \indentlabel{\ifempty{#2}{}{\ignorespaces#1 #2.\quad}}#3}}

\def\boitegrise#1#2{\begin{centerline}{\fcolorbox{black}{shadecolor}{~
    \begin{minipage}[t]{#2}{\vphantom{~}#1\vphantom{$A_{\DS{A_A}}$}}
            \end{minipage}~}}\end{centerline}\medskip}

\theoremstyle{remark}

\theoremstyle{plain}

\def\xyinj{\ar@{^{(}->}}
\def\xysur{\ar@{->>}}

\bigskip

\makeatletter
\def\hlinewd#1{%
\noalign{\ifnum0=`}\fi\hrule \@height #1 %
\futurelet\reserved@a\@xhline}
\makeatother

\newlength\epaisLigne

\usepackage{multirow}

\begin{document}

\baselineskip=16pt

\title{Vogan classes and cells in the unequal parameter case}

\author{{\sc C\'edric Bonnaf\'e}}
\address{
Institut de Math\'ematiques et de Mod\'elisation de Montpellier (CNRS: UMR 5149), 
Universit\'e Montpellier 2,
Case Courrier 051,
Place Eug\`ene Bataillon,
34095 MONTPELLIER Cedex,
FRANCE} 

\makeatletter
\email{cedric.bonnafe@univ-montp2.fr}
\makeatother

%
%


\date{\today}

\thanks{The author is partly supported by the ANR (Project No ANR-12-JS01-0003-01 ACORT)}

\begin{abstract} 
Kazhdan and Lusztig proved that Vogan classes are unions of cells in the equal parameter case. 
We extend this result to the unequal parameter case.
\end{abstract}

\maketitle

\pagestyle{myheadings}

\markboth{\sc C. Bonnaf\'e}{\sc Vogan classes}

Let $(W,S)$ be a Coxeter system and let $\ph : S \to \ZM_{> 0}$ be a weight function. 
To this datum is associated a partition of $W$ into left, right and two-sided cells. 
Determining these partitions is a difficult problem, with deep connections 
(whenever $W$ is a finite or an affine Weyl group) with representations 
of reductive groups, singularities of Schubert cells, geometry of unipotent classes. 

In their original paper, Kazhdan and Lusztig described completely this partition 
whenever $W$ is the symmetric group in terms of the Robinson-Schensted correspondence. 
Their main tool is the so-called $*$-operation. It is defined in any Coxeter group 
(whenever there exists $s$, $t \in S$ such that $st$ has order $3$): they proved 
that it provides some extra-properties of cells, {\it whenever $\ph$ is constant}.
Our aim in this 
paper is to prove that Kazhdan-Lusztig result relating cells and the $*$-operation 
holds in full generality (this result has also been proved independently 
and simultaneously by M. Geck~\cite{geck}). 

\bigskip

\noindent{\bf Commentary.}  In~\cite{bonnafe b}, the author used improperly 
the $*$-operation in the unequal parameter context. The present paper 
justifies {\it a posteriori} what was, at that time, a big mistake!

\bigskip

\noindent{\bf Acknowledgement.} 
I would like to thank M. Geck for pointing out some references and for 
his useful remarks.

\bigskip

\noindent{\bf Notation.}  
We fix in this paper a Coxeter system $(W,S)$ and a totally ordered abelian group $\G$. 
We use an exponential notation for the group algebra $A=\ZM[\G]$:
$$A=\mathop{\oplus}_{\g \in \G} \ZM v^\g,$$
with $v^\g v^{\g'}=v^{\g+\g'}$ for all $\g$, $\g' \in \G$. If $\g_0 \in \G$, we set 
$$\G_{\!\ge \g_0}=\{\g \in \G~|~\g \ge \g_0\},$$
$$\G_{\!> \g_0}=\{\g \in \G~|~\g > \g_0\},$$
$$A_{\!\ge \g_0}=\mathop{\oplus}_{\g \ge \g_0} \ZM v^\g,$$
$$A_{\!> \g_0}=\mathop{\oplus}_{\g > \g_0} \ZM v^\g$$
and similarly for $\G_{\!\le \g_0}$, $\G_{\!< \g_0}$, $A_{\!\le \g_0}$ and $A_{\!< \g_0}$.
We denote by $\overline{\hphantom{x}\vphantom{a}} : A \to A$ the involutive automorphism 
such that $\overline{v^\g}=v^{-\g}$.

\bigskip

\section{Preliminaries}\label{section:invo}

\bigskip

\boitegrise{{\bf Hypothesis and notation.} {\it In this section, and only in this section, 
we fix an $A$-module $\MC$ and we assume that:
\begin{itemize}
\itemth{P1} $\MC$ admits an $A$-basis $(m_x)_{x \in X}$, where $X$ is a {\it poset}. 
We set 
$$\MC_{> 0} = \oplus_{x \in X} A_{>0} m_x.$$
\itemth{P2} $\MC$ admits a semilinear involution $\overline{\hphantom{x}\vphantom{a}} : \MC \to \MC$. 
We set
$$\MC_{\mathrm{skew}} = \{m \in \MC~|~m+\overline{m}=0\}.$$
\itemth{P3} If $x \in X$, then $\DS{\overline{m}_x \equiv m_x 
\mod\Bigl(\mathop{\oplus}_{y < x} A m_y\Bigr)}$
\itemth{P4} If $x \in X$, then the set $\{y \in X~|~y \le x\}$ is finite.
\end{itemize}}\vskip-0.5cm}{0.75\textwidth}

\vskip0.5cm

\def\invotext{$\overline{\hphantom{x}\vphantom{a}}$}

\begin{prop}\label{theo:inf-anti}
The $\ZM$-linear map 
$$\fonctio{\MC_{>0}}{\MC_{\mathrm{skew}}}{m}{m-\overline{m}}$$
is an isomorphism.
\end{prop}

\begin{proof}
First, note that the corresponding result for the $A$-module $A$ itself holds. In other words, 
\equat\label{eq:inf-anti}
\text{\it The map $A_{> 0} \to A_{\mathrm{skew}}$, $a \mapsto a -\overline{a}$ is an isomorphism.}
\endequat
Indeed, if $a \in A_{\mathrm{skew}}$, write $a=\sum_{\g \in \G} r_\g v^\g$, with $r_\g \in R$. Now, 
if we set $a_+=\sum_{\g > 0} r_\g v^\g$, then $a=a_+ - \overline{a}_+$. This shows the surjectivity, 
while the injectivity is trivial.

\medskip

Now, let $\L : \MC_{>0} \to \MC_{\mathrm{skew}}$, $m \mapsto m - \overline{m}$. 
For $\XC \subset X$, we set $\MC^\XC=\oplus_{x \in \XC} A \, m_x$ and $\MC_{>0}^\XC=\oplus_{x \in \XC} A_{>0}\, m_x$. 
Assume that, for all $x \in \XC$ and all $y \in X$ such that $y \le x$, then $y \in \XC$. By hypothesis, 
$\MC^\XC$ is stabilized by the involution ~\invotext. 
Since $X$ is the union of such finite $\XC$ (by hypothesis), it shows that we may, and we will, 
assume that $X$ is finite. Let us write $X=\{x_0,x_1,\dots,x_n\}$ in such a way that, 
if $x_i \le x_j$, then $i \le j$ (this is always possible). For simplifying notation, 
we set $m_{x_i}=m_i$. Note that, by hypothesis, 
$$\overline{m}_i \in m_i + \Bigl(\mathop{\oplus}_{0 \le j < i} A\, m_j\Bigr).\leqno{(*)}$$
In particular, $\overline{m}_0 = m_0$. 

Now, let $m \in \MC_{> 0}$ be such that $\overline{m}=m$ and assume that $m \neq 0$. 
Write $m=\sum_{i=0}^r a_i m_i$, with $r \le n$, $a_i \in A_{>0}$ and $a_r \neq 0$. Then, by hypothesis, 
$$\overline{m} \equiv \overline{a}_r m_r \mod\Bigl(\mathop{\oplus}_{0 \le j <i} A m_j\Bigr).$$
Since $\overline{m}=m$, this forces $\overline{a}_r=a_r$, which is impossible 
(because $a_r \in A_{>0}$ and $a_r \neq 0$). So $\L$ is injective.

\medskip

Let us now show that $\L$ is surjective. So, 
let $m \in \MC_{\mathrm{skew}}$, and assume that $m \neq 0$ (for otherwise there is nothing to prove). 
Write $m=\sum_{i=0}^r a_i m_{x_i}$, with $r \le n$, $a_i \in A$ and $a_r \neq 0$. 
We shall prove by induction on $r$ that there exists $\mu \in \MC$ such that $m=\mu - \overline{\mu}$. 
If $r=0$, then the result follows from~(\ref{eq:inf-anti}) and the fact that $\overline{m}_0=m_0$. 
So assume that $r > 0$. By hypothesis, 
$$m + \overline{m} \equiv (a_r + \overline{a}_r) m_{r} \mod \MC^{\XC_{r-1}},$$
where $\XC_j=\{x_0,x_1,\dots,x_j\}$. 
Since $m + \overline{m}=0$, this forces $a_r \in A_{\mathrm{skew}}$. So, by~(\ref{eq:inf-anti}), 
there exists $a \in A_{> 0}$ such that $a-\overline{a} = a_r$. 
Now, let $m'= m - a m_{r} + \overline{a} \overline{m}_{r}$. Then 
$m' + \overline{m}' = 0$ and $m' \in \oplus_{0 \le j < r} A \, m_j$. So, by the induction hypothesis, 
there exists $\mu' \in \MC_{>0}$ such that $m'=\mu' - \overline{\mu}'$. 
Now, set $\mu = a m_{r} + \mu'$. Then $\mu \in \MC_{>0}$ and $m=\mu-\overline{\mu}=\L(\mu)$, 
as desired.                                                   
\end{proof}

\begin{coro}\label{coro:anti-stable}
Let $m \in \MC$. Then there exists a unique $M \in \MC$ such that 
$$
\begin{cases}
\Mov=M,\\
M \equiv m \mod \MC_{> 0}.\\
\end{cases}
$$
\end{coro}

\begin{proof}
Setting $M=m+\mu$, the problem is equivalent to find $\mu \in \MC_{> 0}$ such that $\overline{m+\mu} = m+\mu$. 
This is equivalent to find $\mu \in \MC_{> 0}$ such that 
$\mu-\overline{\mu} = \mov-m$: since $\mov - m \in \MC_{\mathrm{skew}}$, this problem 
admits a unique solution, thanks to Theorem~\ref{theo:inf-anti}. 
\end{proof}

\bigskip

The Corollary~\ref{coro:anti-stable} can be applied to the $A$-module 
$A$ itself. However, in this case, its proof becomes obvious: 
if $a_\circ = \sum_{\g \in \G} a_\g v^\g$, 
then $a=\sum_{\g \le 0} a_\g v^\g + \sum_{\g > 0} a_{-\g} v^\g$ is the unique 
element of $A$ such that $\aov=a$ and $a \equiv a_\circ \mod A_{>0}$.


\bigskip

\begin{coro}\label{coro:kl-base-triangulaire}
Let $\XC$ be a subset of $X$ such that, if $x \le y$ and $y \in \XC$, then $x \in \XC$. 
Let $M \in \MC$ be such that $\Mov=M$ and $M \in \MC^\XC + \MC_{> 0}$. Then 
$M \in \MC^\XC$. 
\end{coro}

\begin{proof}
Let $M_0 \in \MC^\XC$ be such that $M \equiv M_0 \mod \MC_{>0}$. 
From the existence statement of Corollary~\ref{coro:anti-stable} applied to $\MC^\XC$, there 
exists $M' \in \MC^\XC$ such that $\Mov'=M'$ and 
$M' \equiv M_0 \mod \MC_{> 0}^\XC$. The fact that $M=M' \in \MC^\XC$ 
now follows from the uniquenes statement of Corollary~\ref{coro:anti-stable}.
\end{proof}

\bigskip

\begin{coro}\label{coro:kl-general}
Let $x \in X$. Then there exists a unique element $M_x \in \MC$ 
such that 
$$\begin{cases}
\overline{M}_x = M_x,\\
M_x \equiv m_x \mod \MC_{> 0}.
\end{cases}$$
Moreover, $M_x \in m_x + \mathop{\oplus}_{y < x} A_{> 0} m_y$ and 
$(M_x)_{x \in X}$ is an $A$-basis of $\MC$.
\end{coro}

\begin{proof}
The existence and uniqueness of $M_x$ follow from Corollary~\ref{coro:anti-stable}. 
The statement about the base change follows by applying this existence and uniqueness 
to $\MC^{X_x}$, where $X_x=\{y \in X~|~ y \le x\}$.

Finally, the fact that $(M_x)_{x \in X}$ is an $A$-basis of $\MC$ follows from the fact that the base 
change from $(m_x)_{x \in X}$ to $(M_x)_{x \in X}$ is unitriangular.
\end{proof}

\bigskip

%
%
%
%
%
%
%
%

Corollary~\ref{coro:kl-general} gathers in a single general statement the argument  
given by Lusztig~\cite{lusztig kl} for the construction of the {\it Kazhdan-Lusztig basis} 
of a Hecke algebra (which is different from the argument contained in the original paper by 
Kazhdan and Lusztig~\cite{KL}) and the construction, still due to 
Lusztig~\cite[Theorem~3.2]{lusztig canonical}, of the {\it canonical basis} associated 
with quantum groups.

\section{The main result}\label{section:main}

\medskip

\subsection{Kazhdan-Lusztig basis}
We fix in this paper a {\it weight function} $\ph : S \to \G_{\! > 0}$ 
(i.e. $\ph(s)=\ph(t)$ whenever $s$ and $t$ are conjugate in $W$). We denote by 
$\HC$ the Hecke algebra associated with $(W,S,\ph)$: as an $A$-module, 
$\HC$ admits an $A$-basis $(T_w)_{w \in W}$ and the multiplication is completely 
determined by the following rules:
$$\begin{cases}
T_w T_{w'} = T_{ww'} & \text{if $\ell(ww')=\ell(w)+\ell(w')$,}\\
(T_s - v^{\ph(s)})(T_s+v^{-\ph(s)})=0, & \text{if $s \in S$.}
\end{cases}$$
Here, $\ell : W \to \ZM_{\! \ge 0}$ denote the {\it length function} on $W$. 

We denote by~\invotext~: $\HC \to \HC$ the involutive antilinear automorphism 
of $\HC$ such that
$$\overline{T}_w=T_{w^{-1}}^{-1}.$$
The triple $(\HC,(T_w)_{w \in W},$\invotext$)$ satisfies the properties (P1), (P2), (P3) 
and (P4) of the previous section. Therefore, if $w \in W$, there exists a unique 
element $C_w \in \HC$ such that 
$$C_w \equiv T_w \mod \HC_{\! > 0}$$
(see Corollary~\ref{coro:kl-general}) and $(C_w)_{w \in W}$ is an $A$-basis 
of $\HC$ (see Corollary~\ref{coro:kl-base-triangulaire}), called the 
Kazhdan-Lusztig basis.

\bigskip

\def\pre#1{\leqslant_{#1}}
\def\preplus#1#2{\leqslant_{#1}^{#2}}
\def\prel{\leqslant_{L}}
\def\prelh{\leqslant_{L}^\hb}
\def\prer{\leqslant_{R}}
\def\prelr{\leqslant_{LR}}
\def\rell{\stackrel{L}{\longleftarrow}}
\def\rellh{\stackrel{\hb,L}{\longleftarrow}}
\def\relrh{\stackrel{\hb,R}{\longleftarrow}}
\def\relr{\stackrel{R}{\longleftarrow}}
\def\relplus#1{\stackrel{#1}{\longleftarrow}}
\def\siml{\sim_{L}}
\def\simr{\sim_{R}}
\def\simlr{\sim_{LR}}

\subsection{Cells}
In this context, we define the preorders $\prel$, $\prer$, $\prelr$ and the equivalence relations 
$\siml$, $\simr$ and $\simlr$ as in~\cite{KL} or~\cite{lusztig}. If $C$ is a left cell 
(i.e. an equivalence class for $\siml$) of $W$, we set 
$$\HC^{\!\prel C} = \mathop{\oplus}_{w \prel C} A~C_w,\qquad
\HC^{\!<_L C} = \mathop{\oplus}_{w <_L C} A~C_w$$
$$M(C)=\HC^{\!\prel C}/\HC^{\!<_L C}.\leqno{\text{and}}$$
By the very definition of the preorder $\prel$, $\HC^{\!\prel C}$ and $\HC^{\!<_L C}$ 
are left ideals of $\HC$, so $M(C)$ inherits a structure of $\HC$-module. If $w \in C$, we denote by 
$c_w$ the image of $C_w$ in $M(C)$: then $(c_w)_{w \in C}$ is an $A$-basis of $M(C)$.

\bigskip

\subsection{Parabolic subgroups} 
If $I \subset S$, we set $W_I=\langle I \rangle$: it is a standard parabolic subgroup 
of $W$ and $(W_I,I)$ is a Coxeter system. We also set
$$\HC_I=\mathop{\oplus}_{w \in W_I} A~T_w.$$
It is a subalgebra of $\HC$, naturally isomorphic to the Hecke algera 
associated with $(W_I,I,\ph_I)$, where $\ph_I : I \to \ZM_{\! > 0}$ denotes the restriction of $\ph$.

We denote by $X_I$ the set of elements $x \in W$ which have minimal length 
in $xW_I$: it is well-known that the map $X_I \to W/W_I$, $x \mapsto xW_I$ is bijective and that
\eqna
X_I&=&\{x \in W_I~|~\forall~s \in I,~\ell(xs) > \ell(x)\} \\
&=& \{x \in W_I~|~\forall~w \in W_I,~\ell(xw)=\ell(x)+\ell(w)\}.
\endeqna
As a consequence, the right $\HC_I$-module $\HC$ is free (hence flat) 
with basis $(T_x)_{x \in X_I}$. 
This remark has the following consequence (in the next lemma, if $E$ is a subset of $\HC$, 
then $\HC E$ denotes the left ideal generated by $E$):

\bigskip

\begin{lem}\label{lem:platitude}
If $\IG$ and $\IG'$ are left ideals of $\HC_I$ such that $\IG \subset \IG'$, then:
\begin{itemize}
\itemth{a} $\HC\IG=\oplus_{x \in X_I} T_x \IG$.

\itemth{b} The natural map $\HC \otimes_{\HC_I} \IG \to \HC\IG$ is an isomorphism of $\HC$-modules.

\itemth{c} The natural map $\HC \otimes_{\HC_I} (\IG'/\IG) \to \HC \IG'/\HC \IG$ is an isomorphism. 
\end{itemize}
\end{lem}

\bigskip

We will now recall results from Geck~\cite{geck induction} about the parabolic induction of cells. 
First, it is clear that $(C_w)_{w \in W_I}$ is the Kazhdan-Lusztig basis 
of $\HC_I$. We can then define a preorder $\prel^I$ and its associated equivalence class $\siml^I$ 
on $W_I$ in the same way as $\prel$ and $\siml$ are defined for $W$. 
If $w \in W$, then there exists a unique $x \in X_I$ and a unique 
$w' \in W_I$ such that $w=xw'$: we then set
$$G_w^I=T_x C_{w'}.$$
Finally, if $C$ is a left cell in $W_I$, then we define the left $\HC_I$-module $M^I(C)$ similarly 
as $M(C')$ was defined for left cells $C'$ of $W$.

\bigskip

\begin{theo}[Geck]\label{theo:geck}
Let $E$ be a subset of $W_I$ such that, if $x \in E$ and if $y \in W_I$ is such that $y \prel^I x$, then 
$y \in E$. Let $\IG=\oplus_{w \in E} A~C_w$. Then
$$\HC \IG = \mathop{\oplus}_{w \in X_I \cdot E} A~G_w^I = \mathop{\oplus}_{w \in X_I \cdot E} A~C_w.$$
Moreover, the transition matrix between the $A$-basis $(C_w)_{w \in X_I \cdot E}$ 
and the $A$-basis $(G_w^I)_{w \in X_I \cdot E}$ is unitriangular (for the Bruhat order) 
and its non-diagonal entries belong to $A_{>0}$.
\end{theo}

\bigskip

\begin{coro}[Geck]\label{coro:geck}
We have:
\begin{itemize}
\itemth{a} $\prel^I$ and $\siml^I$ are just the restriction of $\prel$ and $\siml$ to $W_I$ 
(and so we will use only the notation $\prel$ and $\siml$).

\itemth{b} If $C$ is a left cell in $W_I$, then $X_I \cdot C$ is a union of left cells of $W$.
\end{itemize}
\end{coro}

\bigskip

\section{Generalized $*$-operation}
\def\dotcup{\hskip1mm\dot{\cup}\hskip1mm}

\medskip

\boitegrise{{\bf Hypothesis and notation.} 
{\it We fix in this section, and only in this section, 
a subset $I$ of $S$, two left cells $C_1$ and $C_2$ of $W_I$, 
and we assume that:
\begin{itemize}
\itemth{V1} There exists a bijection $\s : C_1 \to C_2$ such 
that the $A$-linear map $M^I(C_1) \to M^I(C_2)$, $c_w \mapsto c_{\s(w)}$ is in fact $\HC_I$-linear.
\itemth{V2} If $\{i,j\}=\{1,2\}$, then $\{w \in W_I~|~w \in C_i$ and $w <_L C_j\}=\vide$.
\end{itemize}
We set $E_0=\{w \in W_I~|~w <_L C_1$ or $w <_L C_2\}$, $E_i=X_0 \dotcup C_i$ for $i \in \{1,2\}$
and
$$\HC_I^{(i)}=\mathop{\oplus}_{w \in E_i} A C_w$$
for $i \in \{0,1,2\}$.}}{0.75\textwidth}

\bigskip

\noindent{\sc Remark - } If $W$ is finite and if we assume that Lusztig's 
Conjectures~\cite[Conjectures~P1~to~P15]{lusztig} hold for $(W_I,I,\ph_I)$, 
then (V2) is a consequence of (V1).\finl

\bigskip

Note that $\HC_I^{(i)}$ is a left $\HC_I$-module for $i \in \{0,1,2\}$. By (V2), $\HC_I^{(i)}/\HC_I^{(0)}$ 
is a left $\HC_I$-module isomorphic to $M^I(C_i)$ (for $i \in \{1,2\}$) 
so it admits an $A$-basis $(c_w)_{w \in C_i}$. By (V1), the map $\s$ induces an isomorphism of $\HC_I$-modules
$$\HC_I^{(1)}/\HC_I^{(0)} \longisom \HC_I^{(2)}/\HC_I^{(0)}.\leqno{(\clubsuit)}$$
Let $\s^L : X_I \cdot C_1 \longisom X_I \cdot C_2$ denote the bijection induced by $\s$ 
(i.e. $\s^L(xw)=x\s(w)$ if $x \in X_I$ and $w \in C_1$). By Lemma~\ref{lem:platitude} 
and Theorem~\ref{theo:geck}, 
$$\HC \HC_I^{(i)} = \mathop{\oplus}_{w \in X_I \cdot E_i} A~G_w^I=
\mathop{\oplus}_{w \in X_I \cdot E_i} A~C_w.\leqno{(\diamondsuit)}$$
Now, if $i \in \{1,2\}$ and $w \in X_I \cdot C_i$, we denote by $g_w^I$ (respectively $\cb_w$) 
the image of $G_w^I$ (respectively $C_w$) in the quotient 
$\HC \HC_I^{(i)}/\HC \HC_I^{(0)}$. By Lemma~\ref{lem:platitude}, the isomorphism 
$(\clubsuit)$ induces an isomorphism of $\HC$-modules 
$$\s_* : \HC \HC_I^{(1)}/\HC \HC_I^{(0)} \longisom \HC \HC_I^{(2)}/\HC \HC_I^{(0)}$$
which is defined by
$$\s_*(g_w^I)=g_{\s^L(w)}^I$$
for all $w \in X_I \cdot C_1$. The key result of this section if the following one:

\bigskip

\begin{theo}\label{theo:main}
If $w \in X_I \cdot C_1$, then $\s_*(\cb_w)=\cb_{\s^L(w)}$.
\end{theo}

\bigskip

\begin{proof}
Let $i \in \{1,2\}$. For simplification, we set $M[i]=\HC \HC_I^{(i)}/\HC \HC_I^{(0)}$. 
By $(\diamondsuit)$, $(g_w^I)_{w \in X_I \cdot C_i}$ and $(\cb_w)_{w \in X_I \cdot C_i}$ are both 
$A$-bases of $M[i]$. We set 
$$M[i]_{\! > 0} = \mathop{\oplus}_{w \in X_I \cdot C_i} A_{\! > 0}~g_w^I.$$
By Geck's Theorem, 
$$\cb_w \equiv g_w^I \mod M[i]_{\! > 0}.\leqno{(\heartsuit)}$$
Moreover, the antilinear involution ~\invotext~ on $\HC$ stabilizes $\HC\HC_I^{(i)}$ and 
$\HC\HC_I^{(0)}$ so it induces an antilinear involution, still denoted by~\invotext, 
on $M[i]$. It is also clear that the isomorphism $\s_* : M[1] \to M[2]$, 
$g_w^I \mapsto g_{\s^L(w)}^I$ satisfies $\s_*(\overline{m})=\overline{\s_*(m)}$ 
and $\s_*(M[1]_{\! > 0})=M[2]_{\! > 0}$. 
Therefore, it follows from $(\heartsuit)$ that, if $w \in X_I \cdot C_1$, then
$$
\begin{cases}
\overline{\s_*(\cb_w)}=\s_*(\cb_w),\\
\s_*(\cb_w) \equiv g_{\s^L(w)}^I \mod M[2]_{\! > 0}.
\end{cases}
$$
But it follows again from Geck's Theorem that the datum $(M[i],(g_w^I)_{w \in X_I \cdot C_i},\text{\invotext})$ 
satisfies the properties (P1), (P2), (P3) and (P4) of \S\ref{section:invo}. So
$\s_*(\cb_w)=\cb_{\s^L(w)}$ by Corollary~\ref{coro:kl-general}.
\end{proof}

\bigskip

We can now state the main consequence of Theorem~\ref{theo:main}. We first need a notation: 
if $C_1 \neq C_2$ (which is the interesting case...), we extend $\s^L$ to an involution 
of the set $W$, by setting 
$$\s^L(w)=
\begin{cases}
w & \text{if $w \not\in X_I \cdot C_1 \dotcup X_I \cdot C_2$,}\\
\s^L(w) & \text{if $w \in X_I \cdot C_1$,}\\
(\s^L)^{-1}(w) & \text{if $w \in X_I \cdot C_2$.}\\
\end{cases}
$$
Note that $\s^L : W \to W$ is an involution.

\bigskip

\begin{theo}\label{theo:vogan}
Let $w$, $w' \in W$. Then $w \siml w'$ if and only if 
$\s^L(w) \sim \s^L(w')$. 
\end{theo}

\bigskip

\begin{proof}
First, let us write
$$C_xC_y=\sum_{z \in W} h_{x,y,z} C_z,$$
where $h_{x,y,z} \in A$.

Now, assume that $w \siml w'$. 
According to Corollary~\ref{coro:geck}(b), there exists a unique 
cell $C$ in $W_I$ such that $w$, $w' \in X_I\cdot C$. If $C \not\in \{C_1,C_2\}$, 
then $\s^L(w)=w$ and $\s^L(w')=w'$, so $\s^L(w) \siml \s^L(w')$. So we may assume that 
$C \in \{C_1,C_2\}$. Since $\s^L$ is involutive, we may assume that $C=C_1$. Therefore, 
$w$, $w' \in X_I \cdot C_1$. 

By the definition of $\prel$ and $\siml$, there exists four sequences $x_1$,\dots, $x_m$, $y_1$,\dots, $y_n$, 
$w_1$,\dots, $w_m$, $w_1'$,\dots, $w_n'$ such that:
$$
\begin{cases}
w_1=w, w_m=w',\\
w_1'=w',w_n'=w,\\
\forall~i \in \{1,2,\dots,m-1\},~h_{x_i,w_i,w_{i+1}} \neq 0,\\
\forall~j \in \{1,2,\dots,n-1\},~h_{y_j,w_j',w_{j+1}'} \neq 0.\\
\end{cases}
$$
Therefore, $w=w_1 \prel w_2 \prel \cdots \prel w_m=w'=w_1' \prel w_2' \prel \cdots \prel w_n'=w$ and so 
$w=w_1 \siml w_2 \siml \cdots \siml w_m=w'=w_1' \siml w_2' \siml \cdots \siml w_n'=w$. 
Again by Corollary~\ref{coro:geck}(b), $w_i$, $w_j' \in X_I\cdot C_1$. So it follows from 
Theorem~\ref{theo:main} that $h_{x,\s^L(w_i),\s^L(w_{i+1})}=h_{x,w_i,w_{i+1}}$ 
and $h_{x,\s^L(w_j'),\s^L(w_{j+1}')}=h_{y_j,w_j',w_{j+1}'}$ for all $x \in W$. Therefore,
$$
\begin{cases}
\forall~i \in \{1,2,\dots,m-1\},~h_{x_i,\s^L(w_i),\s^L(w_{i+1})} \neq 0,\\
\forall~j \in \{1,2,\dots,n-1\},~h_{y_j,\s^L(w_j'),\s^L(w_{j+1}')} \neq 0.\\
\end{cases}
$$
It then follows that 
\eqna
\s^L(w)=\s^L(w_1) \prel \s^L(w_2) \prel \cdots &\prel\!\!\! & \s^L(w_m)=\s^L(w')=\s^L(w_1') \\
&\prel\!\!\!& \s^L(w_2') \prel \cdots \prel \s^L(w_n')=\s^L(w),
\endeqna
and so $\s^L(w) \siml \s^L(w')$, as expected.
\end{proof}

\bigskip

\begin{coro}\label{coro:vogan-plus}
Let $C$ be a left cell of $W$. Then $\s^L(C)$ is a left cell of $W$ and 
the $A$-linear map $M(C) \to M(\s^L(C))$, $c_w \mapsto c_{\s^L(w)}$ is an isomorphism 
of $\HC$-modules.
\end{coro}

\bigskip

\begin{proof}
This follows immediately from Theorems~\ref{theo:main} and~\ref{theo:vogan}.
\end{proof}

\bigskip

\section{Generalized Vogan classes}

\medskip
\def\etoile#1{\bigstar_{\! #1}}

\subsection{Dihedral parabolic subgroups}
Let $\EC_\ph$ denote the set of pairs $(s,t)$ of elements of $S$ such that 
one of the following holds:
\begin{itemize}
\itemth{O} $st$ has odd order $\ge 3$; or

\itemth{E} $\ph(s) < \ph(t)$ and $st$ has even order $\ge 4$.
\end{itemize}
We fix in this subsection a pair $(s,t) \in \EC_\ph$. 
Let $w_{s,t}$ denote the longest element of $W_{s,t}$.
We then set:
$$R_s=\{w \in W_{s,t}~|~\ell(ws) < \ell(w)~\text{and}~\ell(wt) > \ell(w)\} = (W_{s,t} \cap X_t)\setminus X_s$$
$$R_t=\{w \in W_{s,t}~|~\ell(ws) > \ell(w)~\text{and}~\ell(wt) < \ell(w)\} = (W_{s,t} \cap X_s)\setminus X_t.\leqno{\text{and}}$$
If we are in the case (O), we then set
$$\G_s=R_s\qquad\text{and}\qquad \G_t=R_t$$
while, if we are in the case (E), we set
$$\G_s=R_s \setminus \{s\}\qquad \G_t=R_t \setminus \{w_{s,t} s\}.$$
Finally, if we are in the case (O), we $\etoile{s,t} : \G_s \to \G_t$, $w \mapsto w_{s,t} w$ 
while, in the case (E), we set $\etoile{s,t} : \G_s \to \G_t$, $w \mapsto ws$.
Then, by~\cite[\S{7}]{lusztig}, we have:

\bigskip

\begin{lem}\label{lem:st}
If $(s,t) \in \EC_\ph$, then 
$\G_s$ and $\G_t$ are two left cells of $W_{s,t}$ and $\etoile{s,t} : \G_s \to \G_t$ 
is a bijection which satisfies the properties {\rm (V1)} and {\rm (V2)} of \S\ref{section:main}.
\end{lem}

\bigskip

So $\etoile{s,t}$ induces a bijection $\etoile{s,t}^L : W \to W$ and, according to 
Theorem~\ref{theo:vogan}, the following holds:

\bigskip

\begin{coro}\label{coro:star}
If $(s,t) \in \EC_\ph$ and if $C$ is a left cell of $W$, then $C'=\etoile{s,t}^L(C)$ is also 
a left cell and the map $M(C) \to M(C')$, $c_w \mapsto c_{\etoile{s,t}^L(w)}$ is an isomorphism 
of $\HC$-modules.
\end{coro}

\bigskip

\noindent{\sc Remark - } The bijection $\etoile{s,t}^L$ is called the $*$-operation and is usually 
denoted by $w \mapsto m^*$.\finl

\bigskip

\subsection{Generalized Vogan classes}
Let $\VC_\ph$ be the group of bijections of $W$ generated by all the $\etoile{s,t}^L$, 
where $(s,t)$ runs over $\EC_\ph$. We will call it the {\it left Vogan group} ({\it associated with $\ph$}). 
Let $\PC(S)$ denotes the set of subsets of $S$ and, if $w \in W$, we set
$$\RC(w) =\{s \in S~|~\ell(ws) < \ell(w)\}.$$
It is called the {\it right descent set} of $w$. 
It is well-known that the map $\RC : W \to \PC(S)$ is constant on left cells~\cite[Lemma~8.6]{lusztig}.

\bigskip

\noindent{\sc Example - } It can be checked by using computer computations in {\tt GAP} that 
$$|\VC_\ph|=2^{40}\cdot 3^{20} \cdot 5^8 \cdot 7^4 \cdot 11^2$$
whenever $(W,S)$ is of type $H_4$.\finl

\bigskip

Now, let $\maps(\VC_\ph,\PC(S))$ denote the set of maps $\VC_\ph \to \PC(S)$. Then, to each 
$w \in W$, we associate the map $\tau_w^\ph \in \maps(\VC_\ph,\PC(S))$ which is defined by
$$\tau_w^\ph(\s)=\RC(\s(w))$$
for all $\s \in \VC_\ph$. The fiber of the map $\tau^\ph : W \to \maps(\VC_\ph,\PC(S))$ 
are called the {\it generalized Vogan left classes}. In other words, two elements $x$ and $y$ of $W$ 
lie in the same generalized Vogan left class if and only if
$$\forall~\s \in \VC_\ph,~\RC(\s(x))=\RC(\s(y)).$$
It follows from Corollary~\ref{coro:star} that:

\bigskip

\begin{theo}\label{theo:tau-invariant}
Generalized Vogan left classes are unions of left cells.
\end{theo}

\bigskip

%
%
%

\subsection{Knuth classes}
Let $s \in S$. We now define a permutation $\kappa_s^\ph$ of $W$ as follows:
$$\kappa_s^\ph(w)=
\begin{cases}
sw & \text{if there exists $t \in S$ such that $tw < w < sw < tsw$ and $\ph(s) \le \ph(t)$,}\\
sw & \text{if there exists $t \in S$ such that $tsw < sw < w < tw$ and $\ph(s) \le \ph(t)$,}\\
w  & \text{otherwise.}
\end{cases}$$
Then $\kappa_s^\ph$ is an involution of $W$. We denote by $\KC_\ph$ the group of permutations 
of $W$ generated by the $\kappa_s^\ph$, for $s \in S$. A {\it Knuth left class} is 
an orbit for the group $\KC_\ph$. The following result is well-known~\cite{lusztig kl}:

\bigskip

\begin{prop}\label{prop:knuth}
Every left cell is a union of Knuth left classes.
\end{prop}

\bigskip

\subsection{Knuth classes and Vogan classes} 
If $w \in W$, we set
$$\LC(w)=\{s \in S~|~\ell(sw) < \ell(w)\}.$$
It is called the {\it left descent set} of $w$. If $\s \in \VC_\ph$, then
\equat\label{eq:left-descent}
\LC(\s(w))=\LC(w).
\endequat
\begin{proof}
We only need to prove the result whenever $\s=\etoile{s,t}^L$ for some $(s,t) \in \EC_\ph$. 
Now, write $w=xw'$ with $w \in X_{s,t}$ and $w' \in W_{s,t}$ and let $u \in S$. 
Then $\s(w)=x\s(w')$. By Deodhar's Lemma, two cases may occur:

$\bullet$ If $ux \in X_I$, then $u \in \LC(w)$ (or $\LC(\s(w))$) if and only if $u \in \LC(x)$. 
So $u \in \LC(w)$ if and only if $u \in \LC(\s(w))$, as desired.

$\bullet$ If $ux \not\in X_I$, then $ux=xv$, for some $v \in \{s,t\}$. 
Therefore, $u \in \LC(w)$ (respectively $u \in \LC(w)$) 
if and only if $v \in \LC(w')$ (respectively $v \in \LC(\s(w'))$). But it is easy to check 
directly in the dihedral group $W_{s,t}$ that $\LC(w')=\LC(\s(w'))$. So 
again $u \in \LC(w)$ if and only if $u \in \LC(\s(w))$, as desired.
\end{proof}

\bigskip

\begin{prop}\label{prop:k-v}
If $C$ is a Knuth left class and if $\s \in \VC_\ph$, then $\s(C)$ is also a Knuth left class.
\end{prop}

\bigskip

\begin{proof}
It is sufficient to show that, if $s \in S$, if $(t,u) \in \EC_\ph$ and if $w \in W$, 
then $\etoile{t,u}^L(w)$ and $\etoile{t,u}^L(\kappa_s^\ph(w))$ 
are in the same Knuth left class. If $\kappa_s^\ph(w)=w$, then this is obvious. So we may 
(and we will) assume that $\kappa_s(w) \neq w$. Therefore, there exists $s' \in S$ 
such that $s'w < w < sw < s'sw$ or $s'sw < sw < w < s'w$, and $\ph(s) \le \ph(s')$. 
So $\kappa_s^\ph(w)=sw$ and, by replacing if necessary $w$ by $sw$, we may assume that 
$s'w < w < sw < s'sw$. We write $\s=\etoile{t,u}^L$ and $w=xw'$, with $w' \in W_I$. 
Two cases may occur:

\medskip

\noindent{\it First case: assume that $sx \in X_I$.}  
Then $\s(sw)=s\s(w)$. By~(\ref{eq:left-descent}), we have
$$s' \s(w) < \s(w) < s\s(w) = \s(sw) < s'\s(sw) = s's\s(w).$$
So $\kappa_s(\s(w))=\s(sw)$, so $\s(w)$ and $\s(\kappa_s^\ph(w))$ are in the same Knuth left class.

\medskip

\noindent{\it Second case: assume that $sx \not\in X_I$.}   
Then $sx=xt'$ with $t' \in \{t,u\}$ by Deodhar's Lemma. 
Therefore,
$$s'xw' < xw' < xt'w' < s'xt'w'.$$
This shows that $w' < t'w'$ and $s'x \not\in X_I$. Therefore, again by Deodhar's Lemma, 
we have $s'x=xu'$ for some $u' \in \{t,u\}$. Hence
$$u'w' < w' < t'w' < u't'w'$$
and $t' \neq u'$. So $\{t,u\}=\{t',u'\}$. Moreover, $\ph(s)=\ph(t')$ and $\ph(s')=\ph(u')$. 
In this situation, two cases may occur:

\medskip

$\bullet$ Assume that $tu$ has odd order. In this case, $\ph(t')=\ph(u')$ and so $\ph(s)=\ph(s')$. Moreover, 
$$\s(w)=x w_{t,u} w'\qquad\text{and}\qquad \s(sw)=x w_{t,u} t'w'=x u' w_{t,u}w'=s'\s(w).$$
Therefore,
$$s\s(w) < \s(w) < s'\s(w)=\s(sw) < ss' \s(w),$$
and so $\s(\kappa_s^\ph(w))=\kappa_{s'}^\ph(w)$ since $\ph(s)=\ph(s')$. This shows again that 
$\s(w)$ and $\s(\kappa_s^\ph(w))$ are in the same Knuth left cell.

\medskip

$\bullet$ Assume that $tu$ has even order. Since $\{t',u'\}=\{t,u\}$, $\ph(t) < \ph(u)$ and 
$\ph(t')=\ph(s) \le \ph(s')=\ph(u')$, we have $t'=t$ and $u'=u$. In particular,
$$uw' < w' < tw' < utw'.$$
This shows that $w'$, $tw' \in \G_t \dotcup \G_u$, so $\s(w')=w't$ and $\s(tw')=tw't=t\s(w')$. 
Again, by (\ref{eq:left-descent}),
$$u\s(w') < \s(w') < t\s(w')=\s(tw') < ut \s(w')$$
and so
$$s'\s(w) < \s(w) < s\s(w)=\s(sw) < s's\s(w)$$
and so $\s(sw)=\kappa_s^\ph(\s(w))$, as desired.
\end{proof}

\bigskip

\section{Commentaries}

\medskip

\subsection{} 
Since the map $W \to W$, $w \mapsto w^{-1}$ exchanges left cells and right cells, 
and exchanges left descent sets and right descent sets, all the results of this 
paper can be transposed to results about right cells.

\bigskip

\subsection{} 
As it has been seen in Type $H_4$, the group $\VC_\ph$ can become enormous, even in small rank, 
so it is not reasonable to compute generalized Vogan left classes by computing 
completely the map $\tau^\ph$. Computation can be performed by imitating 
the inductive definition of classical Vogan left classes. With our point-of-view, 
this amounts to start with the partition given by the fibers of the map 
$\RC : W \to \PC(S)$, and to refine it successively using the action of the 
generators of $\VC_\ph$, and to stop whenever the partition does not refine any more. 

More precisely, let $\VC_\ph(k)$ denote the set 
of elements of $\VC_\ph$ which can be expressed as the product of at most 
$k$ involutions of the form $\etoile{s,t}^L$, for $(s,t) \in \EC_\ph$ 
and let $\tau^\ph(k) : W \to \maps(\VC_\ph(k),\PC(S))$ be the map 
obtained in a similar way as $\tau^\ph$. This map can be easily 
computed inductively for small values of $k$, and gives rise 
to a partition $\CC(k)$ of $W$ which is {\it a priori} coarser than the partition 
into generalized Vogan left classes. However, when $\CC(k)=\CC(k+1)$, 
this means that $\CC(k)$ coincides with the partition into generalized Vogan left classes.

For instance, in type $H_4$, this algorithm stops at $k=5$. Computing the generators 
of $\VC_\ph$ takes less than 4 minutes on a very basic computer, while the deduction of 
Vogan classes is then almost immediate.

\end{document}